\documentclass[11pt]{amsart}%
\usepackage{amssymb}
\usepackage{amsfonts}
\usepackage{amsmath}
\usepackage{graphicx}%
\newtheorem{theorem}{Theorem}
\theoremstyle{plain}

\numberwithin{equation}{section}
\begin{document}
\title[Multiple solutions to the Bahri-Coron problem]{Multiple solutions to the Bahri-Coron problem in some domains with nontrivial topology}
\author{M\'{o}nica Clapp}
\address{Instituto de Matem\'{a}ticas, Universidad Nacional Aut\'{o}noma de M\'{e}xico,
Circuito Exterior C.U., 04510 M\'{e}xico D.F., Mexico}
\email{mclapp@matem.unam.mx}
\author{Jorge Faya}
\address{Instituto de Matem\'{a}ticas, Universidad Nacional Aut\'{o}noma de M\'{e}xico,
Circuito Exterior C.U., 04510 M\'{e}xico D.F., Mexico}
\email{jorgefaya@gmail.com}
\thanks{Research partially suported by CONACYT grant 129847 and PAPIIT-DGAPA-UNAM
grant IN106612 (M\'{e}xico)}
\date{February 2012}

\begin{abstract}
We show that in every dimension $N\geq3$ there are many bounded domains
$\Omega\subset\mathbb{R}^{N},$ having only finite symmetries, in which the
Bahri-Coron problem
\[
-\Delta u=\left\vert u\right\vert ^{4/(N-2)}u\text{ \ in }\Omega,\text{
\ \ }u=0\text{ \ on }\partial\Omega,
\]
has a prescribed number of solutions, one of them being positive and the rest
sign changing.

\end{abstract}
\maketitle

\section{Introduction}

We consider the problem%
\[
(\wp_{\Omega})\text{\qquad}\left\{
\begin{array}
[c]{ll}%
-\Delta u=|u|^{2^{\ast}-2}u & \text{in }\Omega,\\
\text{ \ \ \ }u=0 & \text{on }\partial\Omega,
\end{array}
\right.
\]
where $\Omega$ is a bounded smooth domain in $\mathbb{R}^{N}$ and $2^{\ast
}:=\frac{2N}{N-2}$ is the critical Sobolev exponent.

Equations of this type arise in fundamental questions in differential geometry
like the Yamabe problem or the scalar curvature problem.

Problem $(\wp_{\Omega})$ has a rich geometric structure: it is invariant under
the group of M\"{o}bius transformations. This fact causes a lack of
compactness of the associated variational functional, which prevents the
straightforward application of standard variational methods.

It is well known that the existence of a solution depends on the domain.
Pohozhaev's identity \cite{po}, together with the unique continuation of
solutions \cite{k}, implies that $(\wp_{\Omega})$ does not have a nontrivial
solution if $\Omega$ is strictly starshaped. On the other hand, if the domain
is an annulus,%
\[
A=A_{a,b}:=\{x\in\mathbb{R}^{N}:0<a<\left\vert x\right\vert <b\},
\]
Kazdan and Warner \cite{kw} showed that $(\wp_{A})$ has infinitely many radial
solutions. Moreover, if $\Omega$ is invariant under the action of a group $G$
of linear isometries of $\mathbb{R}^{N}$ and every $G$-orbit of $\Omega$ is
infinite, compactness is restored (cf. Theorem \ref{teoPS}\ below), and
standard variational methods provide infinitely many $G$-invariant solutions
to problem $(\wp_{\Omega})$.

The first nontrivial existence result is due to Coron \cite{co}. He showed
that, if $0\notin\Omega$ and $\Omega$ contains an annulus $A_{a,b}$ with $b/a$
large enough, then problem $(\wp_{\Omega})$ has at least one positive
solution. A few years later a remarkable result was obtained by Bahri and
Coron \cite{bc} who showed that problem $(\wp_{\Omega})$ has at least one
positive solution in every domain $\Omega$ having nontrivial reduced homology
with $\mathbb{Z}/2$-coefficients.

Concerning multiplicity, as we already mentioned, $(\wp_{\Omega})$ has
infinitely many solutions if $\Omega$ is $G$-invariant and every $G$-orbit of
$\Omega$ is infinite. On the other hand, several multiplicity results have
been established for domains which are obtained by deleting a thin enough
neighborhood of a certain subset from a given domain, see e.g.
\cite{cgp,cmp,cw1,cw4,mp,mp2,lyy,pa,r}. In particular, for the type of domains
considered by Coron, Ge, Musso and Pistoia \cite{gmp} recently obtained a
strong multiplicity result: they basically showed that, if $0\notin\Omega$ and
$\Omega$ contains an annulus $A_{\varepsilon,b},$ then the number of solutions
to $(\wp_{\Omega})$ becomes arbitrarily large as $\varepsilon\rightarrow0$.
This result requires no symmetries on the domain $\Omega.$ Its proof is based
on the Lyapunov-Schmidt reduction method.

But for domains which are neither highly symmetric nor small perturbations of
a given domain, multiplicity remains largely open. A first result in this
direction was recently established in \cite{cp}. Here we shall extend the main
result in \cite{cp} in a way which provides many new examples of domains
$\Omega$ in which problem $(\wp_{\Omega})$ has a prescribed number of solutions.

We need some notation. Let $O(N)$ be the group of linear isometries of
$\mathbb{R}^{N}$. If $G$ is a closed subgroup of $O(N)$, we denote by
\[
Gx:=\{gx:g\in G\}
\]
the $G$-orbit of $x\in\mathbb{R}^{N}$ and by $\#Gx$ its cardinality. A domain
$\Omega\subset\mathbb{R}^{N}$ is called $G$-invariant if $Gx\subset\Omega$ for
all $x\in\Omega,$ and a function $u:\Omega\rightarrow\mathbb{R}$ is called
$G$-invariant if $u$ is constant on every $Gx.$

Fix a closed subgroup $\Gamma$ of $O(N)$ and a nonempty $\Gamma$-invariant
bounded smooth domain $D$ in $\mathbb{R}^{N}$ such that $\#\Gamma x=\infty
\ $for all $x\in D.$ We prove the following result.

\begin{theorem}
\label{thmmain}There exists an increasing sequence $(\ell_{m})$ of positive
real numbers, depending only on $\Gamma$ and $D$, with the following property:
If $\Omega$ contains $D$ and if it is invariant under the action of a closed
subgroup $G$ of $\Gamma$ such that%
\[
\min_{x\in\Omega}\#Gx>\ell_{m},
\]
then problem $(\wp_{\Omega})$ has at least $m$ pairs of $G$-invariant
solutions $\pm u_{1},\ldots,\pm u_{m}$ such that $u_{1}$ is positive,
$u_{2},\ldots,u_{m}$ change sign, and
\[
\int_{\Omega}\left\vert \nabla u_{k}\right\vert ^{2}\leq\ell_{k}%
S^{N/2}\text{\qquad for every }k=1,\ldots,m,
\]
where $S$ is the best Sobolev constant for the embedding $D^{1,2}%
(\mathbb{R}^{N})\hookrightarrow L^{2^{\ast}}(\mathbb{R}^{N})$.
\end{theorem}

The particular case where $\Gamma=O(N)$ and $D=A_{a,b}$ was established in
\cite{cp}. This situation is, however, quite restrictive, particularly in odd
dimensions. For example, if $N=3,$ then $\min_{x\in A_{a,b}}\#Gx\leq12$ for
every subgroup $G\neq SO(3),O(3),$ cf. \cite{clm}. As we shall see, the number
$\ell_{1}$ goes to infinity as $b/a\rightarrow1.$ Therefore, the main result
in \cite{cp}\ will provide solutions in subdomains of $\mathbb{R}^{3}$ only if
$b/a$ is sufficiently large, which is the case already handled by Coron
\cite{co} and by Ge, Musso and Pistoia \cite{gmp}.

Theorem \ref{thmmain}, on the other hand, provides examples in every dimension
of domains $\Omega$, having only finite symmetries, in which problem
$(\wp_{\Omega})$ has a prescribed number of solutions. Specific examples may
be obtained as follows: let $D_{0}$ be a bounded smooth domain in
$\mathbb{R}^{N-1},$ $N\geq3,$ with $D_{0}\subset\{(x,y)\in\mathbb{R}%
\times\mathbb{R}^{N-2}:x\geq\varepsilon\}$ for some $\varepsilon>0.$ Set
\[
D:=\{(z,x^{\prime})\in\mathbb{C}\times\mathbb{R}^{N-2}\equiv\mathbb{R}%
^{N}:(\left\vert z\right\vert ,y)\in D_{0}\}.
\]
Then $D$ is invariant under the action of the group $\Gamma:=\mathbb{S}^{1}$
of unit complex numbers, acting by $e^{i\theta}(z,x^{\prime}):=(e^{i\theta
}z,x^{\prime}).$ Note that this action is free on $(\mathbb{C}\smallsetminus
\{0\}\mathbb{)}\times\mathbb{R}^{N-2}$, so if $G_{n}:=\{e^{2\pi ik/n}%
:k=0,...,n-1\}$ is the cyclic subgroup of order $n,$ then $\#G_{n}x=n$ for
every $x\in(\mathbb{C}\smallsetminus\{0\}\mathbb{)}\times\mathbb{R}^{N-2}.$
Therefore, for every $n>\ell_{m}$ and every $G_{n} $-invariant bounded smooth
domain $\Omega$ in $\mathbb{R}^{N}$ with
\[
D\subset\Omega\subset(\mathbb{C}\smallsetminus\{0\}\mathbb{)}\times
\mathbb{R}^{N-2},
\]
Theorem \ref{thmmain} yields at least $m$ pairs of solutions to problem
$(\wp_{\Omega}).$

This result supports our belief that multiplicity should hold in
noncontractible domains, as those considered in \cite{bc}. But the proof of
such a general statement is still way out of reach.

The proof of Theorem \ref{thmmain} is variational and it is given in the
following section.

\section{Proof of the main theorem.}

Let $G$ a closed subgroup of $O(N)$. If $\Omega$ is $G$-invariant, the
principle of symmetric criticality \cite{p}\ asserts that the $G$-invariant
solutions of problem $(\wp_{\Omega})$\ are the critical points of the
restriction of the functional
\[
J(u):=\frac{1}{2}\int_{\Omega}\left\vert \nabla u\right\vert ^{2}-\frac
{1}{2^{\ast}}\int_{\Omega}\left\vert u\right\vert ^{2^{\ast}}%
\]
to the space of $G$-invariant functions
\[
H_{0}^{1}(\Omega)^{G}:=\{u\in H_{0}^{1}(\Omega):u(gx)=u(x)\text{ \ for all
}g\in G,\text{ }x\in\Omega\}.
\]

We shall say that $J$ satisfies the Palais-Smale condition $(PS)_{c}^{G}$ in
$H_{0}^{1}(\Omega)$ if every sequence $(u_{n})$ such that
\[
u_{n}\in H_{0}^{1}(\Omega)^{G},\qquad J(u_{n})\rightarrow c,\qquad\nabla
J(u_{n})\rightarrow0,
\]
contains a convergent subsequence.

Let $c_{\infty}:=\frac{1}{N}S^{N/2}$ be the energy of the positive solution
(unique up to translation and dilation) to the problem%
\[
-\Delta u=|u|^{2^{\ast}-2}u,\text{\qquad}u\in D^{1,2}(\mathbb{R}^{N}).
\]
We shall make use the following results.

\begin{theorem}
\label{teoPS}$J$ satisfies condition $(PS)_{c}^{G}$ in $H_{0}^{1}(\Omega)$ for
every%
\[
c<\min_{x\in\overline{\Omega}}(\#Gx)c_{\infty}.
\]
In particular, if $\#Gx=\infty$ for all $x\in\overline{\Omega},$ then $J$
satisfies condition $(PS)_{c}^{G}$ in $H_{0}^{1}(\Omega)$ for every
$c\in\mathbb{R}$.
\end{theorem}

\begin{proof}
See \cite[Corollary 2]{c}.
\end{proof}

\begin{theorem}
\label{thmcp}Let $W$ be a finite dimensional subspace of $H_{0}^{1}%
(\Omega)^{G}.$ If $J$ satisfies condition $(PS)_{c}^{G}$ in $H_{0}^{1}%
(\Omega)$\ for all $c\leq\sup_{W}J$, then $J$ has at least $\dim(W)-1$ pairs
of sign changing critical points $u\in H_{0}^{1}(\Omega)^{G}$ such that
$J(u)\leq\sup_{W}J.$
\end{theorem}

\begin{proof}
See \cite[Theorem 3.7]{cp}.
\end{proof}

\bigskip

\noindent\textbf{Proof of Theorem \ref{thmmain}.}\emph{\qquad}Let
$\mathcal{P}_{1}(D)$ be the set of all nonempty $\Gamma$-invariant bounded
smooth domains contained in $D,$ and define%
\[
\mathcal{P}_{k}(D):=\{(D_{1},\mathcal{\ldots},D_{k}):D_{i}\in\mathcal{P}%
_{1}(D)\text{, \ }D_{i}\cap D_{j}=\emptyset\text{ if }i\neq j\}.
\]
Note that $\mathcal{P}_{k}(D)\neq\emptyset$ for every $k\in\mathbb{N}.$ Since
$\#\Gamma x=\infty$ for all $x\in D_{i},$ Theorem \ref{teoPS} asserts that $J$
satisfies condition $(PS)_{c}^{\Gamma}$ in $H_{0}^{1}(D_{i})$ for all
$c\in\mathbb{R}$. Hence, the mountain pass theorem \cite{ar} yields a
nontrivial least energy $\Gamma$-invariant solution $\omega_{D_{i}}\ $to
problem $(\wp_{D_{i}}).$ We define%
\[
c_{k}:=\inf\left\{
{\textstyle\sum\limits_{i=1}^{k}}
J(\omega_{D_{i}}):(D_{1},\mathcal{\ldots},D_{k})\in\mathcal{P}_{k}(D)\right\}
\text{\qquad and\qquad}\ell_{k}:=c_{\infty}^{-1}c_{k}.
\]
Note that $c_{1}=J(\omega_{D}).$ Since $J(\omega_{D_{i}})\geq c_{\infty},$ we
have that%
\[
c_{k-1}+c_{\infty}\leq%
{\textstyle\sum\limits_{i=1}^{k}}
J(\omega_{D_{i}})
\]
for every $(D_{1},\mathcal{\ldots},D_{k})\in\mathcal{P}_{k}(D)$, $k\geq2.$ It
follows that
\[
c_{k-1}+c_{\infty}\leq c_{k}\text{\qquad and\qquad}\ell_{k-1}+1\leq\ell_{k}.
\]

Let $m\in\mathbb{N}$ and let $\Omega$ be a bounded smooth domain containing
$D,$ which is invariant under the action of a closed subgroup $G$ of $\Gamma$
such that
\begin{equation}
\min_{x\in\Omega}\#Gx>\ell_{m}. \label{hypothesis}%
\end{equation}
Given $\varepsilon\in(0,c_{\infty})$ with $c_{m}+\varepsilon<\left(
\min_{x\in\Omega}\#Gx\right)  c_{\infty},$ we choose $(D_{1},\mathcal{\ldots
},D_{m})\in\mathcal{P}_{m}(D)$ such that%
\[
c_{m}\leq%
{\textstyle\sum\limits_{i=1}^{m}}
J(\omega_{D_{i}})<c_{m}+\varepsilon.
\]
Observe that $\omega_{D_{i}}\in H_{0}^{1}(\Omega)^{G}$ and satisfies%
\begin{equation}
J(\omega_{D_{i}})=\max_{t\geq0}J(t\omega_{D_{i}}). \label{mountainpass}%
\end{equation}
For each $k=1,\ldots,m,$ let $W_{k}$ be the subspace of $H_{0}^{1}(\Omega
)^{G}$ generated by $\{\omega_{D_{1}},\ldots,\omega_{D_{k}}\}$ and
$d_{k}:=\sup_{W_{k}}J.$ Since $D_{i}\cap D_{j}=\emptyset$ if $i\neq j,$ the
intersection of the supports of $\omega_{D_{i}}$ and $\omega_{D_{j}}$ has
measure zero. Therefore $\omega_{D_{i}}$ and $\omega_{D_{j}}$ are orthogonal
in $H_{0}^{1}(\Omega)^{G}$ and, consequently, $\dim W_{k}=k.$ Identity
(\ref{mountainpass}) implies that
\[
d_{k}=\sup_{W_{k}}J\leq%
{\textstyle\sum\limits_{i=1}^{k}}
J(\omega_{D_{i}})<\left(  \min_{x\in\Omega}\#Gx\right)  c_{\infty}.
\]
Then, by Theorem \ref{teoPS}, $J$ satisfies $(PS)_{c}^{G}$ in $H_{0}%
^{1}(\Omega)$ for all $c\leq d_{k},$ so the mountain pass theorem \cite{ar}
yields a positive critical point $u_{1}\in H_{0}^{1}(\Omega)^{G}$ of $J $ such
that $J(u_{1})\leq d_{1}.$ Moreover, applying Theorem \ref{thmcp}\ to each
$W_{k}$, we obtain $m-1$ pairs of sign changing critical points $\pm
u_{2},\ldots,\pm u_{m}$ such that%
\[
J(u_{k})\leq d_{k}\text{\qquad for every }k=1,\ldots,m.
\]
Note that%
\[
d_{k}+(m-k)c_{\infty}\leq%
{\textstyle\sum\limits_{i=1}^{m}}
J(\omega_{D_{i}})<c_{m}+\varepsilon
\]
so, since $\varepsilon\in(0,c_{\infty}),$ we conclude that%
\begin{equation}
J(u_{k})<c_{m}\text{\qquad for every }k=1,\ldots,m-1. \label{claim1}%
\end{equation}
Next, we prove that we may choose $u_{m}$ such that
\begin{equation}
J(u_{m})\leq c_{m}. \label{claim2}%
\end{equation}
Let $\varepsilon_{n}\in(0,c_{\infty})$ be such that $\varepsilon
_{n}\rightarrow0,$ and let $u_{m,n}$ denote the $m$-th critical point obtained
by applying the previous argument with $\varepsilon=\varepsilon_{n}$. Then
$J(u_{m,n})<c_{m}+\varepsilon_{n}.$ If $J(u_{m,n_{0}})\leq c_{m} $ for some
$n_{0},$ we are done. If $J(u_{m,n})>c_{m}$ for all $n$, then $J(u_{m,n}%
)\rightarrow c_{m}.$ Since $\nabla J(u_{m,n})=0$ and $J$ satisfies
$(PS)_{c_{m}}^{G},$ there exists a $u_{m}\in H_{0}^{1}(\Omega)^{G}$ such that,
after passing to a subsequence, $u_{m,n}\rightarrow u_{m}.$ Therefore, $u_{m}$
is a critical point of $J$ with $J(u_{m})=c_{m}.$ Note that $u_{m}$ is
positive if $m=1$ and\ it is sign changing if $m\geq2.$ Moreover,
(\ref{claim1}) implies that $u_{m}\neq u_{k}$ for every $k=1,\ldots,m-1.$ This
proves (\ref{claim2}).

Finally, note that if $\Omega$ satisfies (\ref{hypothesis}) then it also
satisfies%
\[
\min_{x\in\Omega}\#Gx>\ell_{k}\text{\qquad for each \ }k=1,\ldots,m.
\]
So, applying the previous argument to each $k,$ we obtain $k$ pairs
of$\ G$-invariant solutions $\pm u_{1}^{k},\ldots,\pm u_{k}^{k}$ to
$(\wp_{\Omega})$ such that $u_{1}^{k}$ is positive, $u_{2}^{k},\ldots
,u_{k}^{k}$ change sign, and%
\[
J(u_{i}^{k})\leq c_{k}\text{\qquad for every \ }i=1,\ldots,k.
\]
Setting $u_{1}:=u_{1}^{1}$ and choosing $u_{k}\in\{u_{2}^{k},\ldots,u_{k}%
^{k}\}$ with $k\geq2$ inductively, such that $u_{k}\neq u_{i}$ for every
$i=1,\ldots,k-1,$\ we obtain $m$ pairs of $G$-invariant solutions $\pm
u_{1},\ldots,\pm u_{m}$ such that $u_{1}$ is positive, $u_{2},\ldots,u_{m}$
change sign, and%
\[
J(u_{k})\leq c_{k}\text{\qquad for every \ }k=1,\ldots,m,
\]
as claimed. \qed

\end{document}